\documentclass[a4paper,12pt,openany]{amsart}
\usepackage{indentfirst}
\usepackage{fancyhdr}
\usepackage{amssymb}
\usepackage{amsmath} 
\usepackage{latexsym}
\usepackage{amsthm} 
\usepackage{eucal} 
\usepackage{enumerate} 
\usepackage{pdfpages}
\usepackage{mathrsfs}
\usepackage{stackrel}
\usepackage{enumerate}
\usepackage{graphicx}

\theoremstyle{plain}

\newtheorem{theorem}{Theorem}[section]

\newtheorem{cor}[theorem]{Corollary}
\newtheorem{prop}[theorem]{Proposition}

\theoremstyle{definition}
\newtheorem{dfn}[theorem]{Definition}
\newtheorem{rem}[theorem]{Remark}
\newcommand{\ve}{\varepsilon}
\newcommand{\R}{\mathbb{R}}
\newcommand{\N}{\mathbb{N}}

\newcommand{\Q}{\mathcal{Q}}
\newcommand{\E}{\mathscr{E}}

\begin{document}

\title[Linking over cones for the Neumann Fractional $p-$Laplacian]{Linking over cones for the Neumann Fractional $p-$Laplacian}

\author[Dimitri Mugnai]{Dimitri Mugnai}
\author[Edoardo Proietti Lippi] {Edoardo Proietti Lippi}

\address[D. Mugnai]{Department of Ecology and Biology (DEB) \newline\indent
	 Tuscia University\newline \indent
	Largo dell'Universit\`a, 01100 Viterbo, Italy}
\email{dimitri.mugnai@unitus.it}

\address[E. Proietti Lippi]{Department of Mathematics and Computer Science \newline\indent University of Florence\newline\indent
Viale Morgagni 67/A, 50134 Firenze - Italy}
\email{edoardo.proiettilippi@unifi.it}

\begin{abstract}
We consider nonlinear problems governed by the fractional $p-$Laplacian in presence of nonlocal Neumann boundary conditions. We face two problems. First: the $p-$superlinear term may not satisfy the Ambrosetti-Rabinowitz condition. Second, and more important: although the topological structure of the underlying functional reminds the one of the linking theorem, the nonlocal nature of the associated eigenfunctions prevents the use of such a classical theorem. For these reasons, we are led to adopt another approach, relying on the notion of linking over cones.
\end{abstract}
\maketitle

Keywords: fractional $p-$Laplacian, Neumann boundary conditions, linking over cones, lack of Ambrosetti-Rabinowitz condition.

2010AMS Subject Classification: 35A15, 47J30, 35S15, 47G10, 45G05.

\section{Introduction}

In this paper we are concerned with the problem
\begin{equation}\label{plink}
\begin{cases}
(-\Delta)^s_p u =\lambda |u|^{p-2}u + g(x,u) & $ in $ \Omega,
\\
\mathscr{N}_{s,p}u=0  & $ in $ \R^N \setminus \overline{\Omega}.
\end{cases}
\end{equation}
Here $p\in(1,\infty)$, $\Omega$ is a bounded domain with Lipschitz boundary, $\lambda \geq 0$ and $g:\Omega\times \R \to \R$ is a Carath\'eodory function. The novelty of our investigation relies on the fact that we study a quasilinear fractional problem in presence of nonlocal {\em Neumann} boundary conditions, namely we require that
\[
\mathscr{N}_{s,p}u(x):=\int_\Omega \frac{|u(x)-u(y)|^{p-2}(u(x)-u(y))}{|x-y|^{N+ps}}dy=0\]
for every $x\in \R^N \setminus \overline{\Omega}$. As a matter of fact, such a condition is the natural $p-$Neumann boundary condition associated to the operator
\[
(-\Delta)^s_pu(x):=P.V.\int_{\R^N} \frac{|u(x)-u(y)|^{p-2}(u(x)-u(y))}{|x-y|^{N+ps}}dy
\]
for $x\in \Omega$, $P.V.$ being the Cauchy Principal value, see \cite{BMPS,DPROV,mupli} (see also \cite{mazon} for a related case and \cite{warma} for the restricted or regional fractional $p-$Laplacian. See also \cite{mbrs} for a general overlook on nonlocal operators).

Under suitable assumptions on $g$, we will show that problem \eqref{plink} admits solutions. As usual, we shall deal with weak solutions, belonging to a suitable function space. In our case, solutions will be sought in the space
\[
X:=\left\{u:\R^N\to \R  \quad \text{measurable such that }\|u\|<\infty\right\},
\]
where
\[
\|u\|:=\left(\int\int_\Q\frac{|u(x)-u(y)|^p}{|x-y|^{N+ps}}\,dxdy +\|u\|_{L^p(\Omega)}^p\right)^\frac{1}{p},
\]
and $\Q=\R^{2N}\setminus (C\Omega)^2$, $C\Omega=\R^N\setminus \Omega$.

\begin{rem}
It is clear that, when $\Omega$ is sufficiently regular, as in our case, in the integral above we can equally consider $\R^N\setminus \Omega$ or $\R^N\setminus \overline\Omega$.
\end{rem}

We will deal with the following standard
\begin{dfn}
Let $u\in X$. We say that $u$ is a weak solution of \eqref{plink} if
\[
\frac{1}{2}\int\int_\Q \frac{J_p(u(x)-u(y))(v(x)-v(y))}{|x-y|^{N+ps}}\, dxdy= \lambda\int_\Omega |u|^{p-2}uv\,dx+\int_\Omega g(x,u)v\,dx
\]
for every $v\in X$, where $J_p(u(x)-u(y))=|u(x)-u(y)|^{p-2}(u(x)-u(y)) $, provided that the last integral makes sense.
\end{dfn}
Of course, below we will give conditions which ensure that the definition above makes sense.

We observe that we shall consider only the case $\lambda\geq0$. Indeed, the case $\lambda<0$ makes the situation different, since one can apply the Mountain Pass Theorem with the Cerami or with the Palais-Smale condition (see \cite{mupli}). In our case the natural geometric structure for the associated functional is the one of {\em linking over cones}, as introduced in \cite{dela}, for which some suitable topological notions are needed. As usual when dealing with linking structures, it is natural to consider the eigenvalues of the underlying operator; in this case we will employ the sequence of eigenvalues found in \cite{mupli} by using the Fadell-Rabinowitz index. All these preliminary tools will be recalled in Section \ref{secback} below. We also recall that the use of linking theorems for fractional operators with {\em Dirichlet boundary conditions} has already appeared in related situations (see \cite{pp} and \cite{sv}).

As for the nonlinear source, in Section \ref{seclinking} we assume that $g$ has $p-$superlinear growth and satisfies different sets of assumptions: in the first case, we will assume that $g$ satisfies the usual Ambrosetti-Rabinowitz condition, while in the second case we will exploit a different general assumption, introduced in \cite{MP}. We remark that in both cases we encounter the difficulty of determining the topological structure of the associated functional, while in the second case we have the additional complication related to the proof of the Cerami condition. Finally, in Section \ref{secsaddle} we consider the case in which $g$ has $p-$linear growth.

As a matter of fact, there are two examples with $p=2$ that are covered by our results and which explain the nature of our results better:
\[
\begin{cases}
(-\Delta)^s u = \lambda u +|u|^{q-2}u\quad $ in $ \Omega,
\\
\mathscr{N}_{s,2}u=0  \quad \quad $ in $ \R^N \setminus \overline{\Omega},
\end{cases}
\]
with $q>2$ and $q<\frac{2N}{N-2s}$ if $N>2s$, and
\begin{equation}\label{es2}
\begin{cases}
(-\Delta)^s u = \lambda u +f(x)\quad $ in $ \Omega,
\\
\mathscr{N}_{s,2}u=0  \quad \quad $ in $ \R^N \setminus \overline{\Omega},
\end{cases}
\end{equation}
with $\lambda<0$ and $f\in L^2(\Omega)$. For the first problem the idea is to apply a standard Linking Theorem, while in the second case the variational structure is the one of the classical Weierstrass Theorem. In our results the first situation is widened to cover the quasilinear form  of the fractional $p-$Laplacian, which doesn't let us apply the classical Linking theorem directly, since the nonlinear operator $(-\Delta)^s_p$ does not have linear eigenspaces; thus, the use of Linking over cones provides an original opportunity, see \cite{dela}, \cite{liu}, \cite{28}, \cite{29} for related cases in the local situation.

. Moreover, the possibility of treating nonlinear terms non verifying the classical Abrosetti-Rabinowitz condition, makes our results new also in the easier case $p=2$. On the other hand, the easy situation described in problem \eqref{es2} is enlarged to cover quasilinear problems where  a nonlinear term is allowed to be not far from 0, as $\lambda$ is in \eqref{es2} (see Theorem \ref{thsad}).

\section{Background}\label{secback}
First we recall some notions regarding the eigenvalues of fractional $p-$Laplacian, see \cite{BMPS} and \cite{mupli}.
Consider the nonlinear eigenvalue problem
\begin{equation}\label{probla}
\begin{cases}
(-\Delta)^s_p u = \lambda |u|^{p-2}u \quad $ in $ \Omega,
\\
\mathscr{N}_{s,p}u=0  \quad \quad $ in $ \R^N \setminus \overline{\Omega},
\end{cases}
\end{equation}
with $\lambda \in \R$. As usual, if \eqref{probla} admits a weak solution we say that $\lambda$ is an eigenvalue 
of $(-\Delta)^s_p$ with $p-$Neumann boundary conditions. So, there exists a sequence $\lambda_m$ of eigenvalues defined as 
\begin{equation}\label{lm}
\begin{aligned}
\lambda_m:=\inf \left\lbrace \sup_{u\in A}\int\int_\Q \frac{|u(x)-u(y)|^p}{|x-y|^{N+ps}}\, dxdy
\right.&: \, A\subseteq M , A \text{ is symmetric,}\\
&\left. \text{ compact and } i(A) \geq m \right\rbrace,
\end{aligned}
\end{equation}
where $i$ is the $\mathbb{Z}_2$-cohomological index of Fadell and Rabinowitz (see \cite{fara}) and
\[
M:=\left\lbrace u\in X :\, \int_\Omega |u|^p\,dx=1 \right\rbrace.
\]
Notice that $\lambda_1=0$ is the first (simple) eigenvalue with associated eigenspace made of constant functions (see \cite{mupli}).

For each $\lambda_m$, we can define the cones
\begin{equation}\label{c-}
C_m^-:=\left\lbrace u\in X :\, \int\int_\Q \frac{|u(x)-u(y)|^p}{|x-y|^{N+ps}}\, dxdy \leq \lambda_m \int_\Omega |u|^p\,dx \right\rbrace
\end{equation}
\begin{equation}\label{c+}
C^+_m:=\left\lbrace u\in X :\, \int\int_\Q \frac{|u(x)-u(y)|^p}{|x-y|^{N+ps}}\, dxdy \geq \lambda_{m+1} \int_\Omega |u|^p\,dx \right\rbrace.
\end{equation}
For further use, we also introduce the notation
\[
[u]=\left(\int\int_\Q \frac{|u(x)-u(y)|^p}{|x-y|^{N+ps}}\, dxdy\right)^{1/p},
\]
which is closely related to the fractional Gagliardo seminorm.

Now we recall some notions on linking sets and Alexander-Spanier cohomology, referring to \cite{dela}.

\begin{dfn}
Let $D,S,A,B$ be four subsets of a metric space $X$ with $S \subseteq D$ and $B\subseteq A$. We say that $(D,S)$ \textit{links} $(A,B)$,
if $S\cap A= B\cap D =\emptyset$ and, for every deformation $\eta:D\times[0,1] \to X\setminus B$ with 
$\eta(S\times[0,1])\cap A=\emptyset$, we have that $\eta(D\times \{1\})\cap A\neq \emptyset$.
\end{dfn}

To prove the existence of critical points we will use  a particular case of 
\cite[Theorem 3.1]{fri}. A smooth version of such a result was already stated in \cite[Theorem 2.2]{dela} under the validity of the Palais--Smale condition. However, the key point in the proof of \cite[Theorem 3.1]{fri} is the possibility of defining deformations between sublevels, as it is possible under the validity of the Cerami condition. For this reason we recall that $f$ satisfies the $(C)_c$ condition, $c\in \R$, if
\begin{center}
for every $(u_n)_n$ such that
$f(u_n)\to c$ and $(1+\|u_n\|) f'(u_n) \to 0 $ in $X'$, then, up to a subsequence, $u_n \to u$ in $X$.
\end{center}
Hence, we will need the following version of \cite[Theorem 3.1]{fri}:
\begin{theorem}\label{critp}
Let $X$ be a complete Finsler manifold of class $C^1$ and let $f:X\to \R$ be a function of class $C^1$.
Let $D,S,A,B$ be four subsets of $X$, with $S \subseteq D$ and $B\subseteq A$, such that $(D,S)$ links $(A,B)$ 
and such that 
\[
\sup_S f< \inf_A f, \quad \quad \sup_D f< \inf_B f
\]
$($with $\sup \emptyset=-\infty$ and $\inf \emptyset=+\infty)$. Define 
\[
c=\inf_{\eta \in \mathcal{N}}\sup f(\eta(D\times \{1\})),
\]
where $\mathcal{N}$ is the set of deformations $\eta:D\times[0,1] \to X\setminus B$ with 
$\eta(S\times[0,1])\cap A=\emptyset$. Then we have 
\[
\inf_A f\leq c \leq \sup_D f.
\]
Moreover, if $f$ satisfies $(C)_c$, then $c$ is a critical value of $f$.
\end{theorem}

\begin{dfn}
Let $D,S,A,B$ be four subsets of $X$ with $S \subseteq D$ and $B\subseteq A$; let $m$ be a nonnegative integer and let
$\mathbb{K}$ be a field. We say that $(D,S)$ \textit{links} $(A,B)$ \textit{cohomologically in dimension $m$ over} $\mathbb{K}$
if $S\cap A= B\cap D =\emptyset$ and the restriction homomorphism 
$H^m(X\setminus B,X\setminus A;\mathbb{K})\to H^m(D,S;\mathbb{K})$ is not identically zero.
\end{dfn}

The geometry we are interested in is described by the following
\begin{theorem}[\cite{dela}, Theorem 2.8] \label{geom}
Let $X$ be a real normed space and let ${\mathcal C}_-$, ${\mathcal C}_+$ be two cones such that ${\mathcal C}_+$ is closed in $X$, 
${\mathcal C}_- \cap {\mathcal C}_+=\lbrace 0 \rbrace$ and such that $(X,{\mathcal C}_-\setminus \lbrace 0 \rbrace)$ links ${\mathcal C}_+$ cohomologically in dimension
$m$ over $\mathbb{K}$. Let $r_-,r_+>0$ and let
\[
D_-=\lbrace u\in {\mathcal C}_-:\,\|u\|\leq r_-\rbrace, \quad \quad S_-=\lbrace u\in {\mathcal C}_-:\,\|u\|= r_-\rbrace, 
\]
\[
D_+=\lbrace u\in {\mathcal C}_+:\,\|u\|\leq r_+\rbrace, \quad \quad S_+=\lbrace u\in {\mathcal C}_+:\,\|u\|= r_+\rbrace.
\]
Then the following facts hold:
\begin{itemize}
\item[(a)] $(D_-,S_-)$ links ${\mathcal C}_+$ cohomologically in dimension $m$ over $\mathbb{K}$;

\item[(b)]  $(D_-,S_-)$ links $(D_+,S_+)$ cohomologically in dimension $m$ over $\mathbb{K}$;
\end{itemize}

Moreover, let $e\in X$ with $-e\notin {\mathcal C}_-$, let
\[
Q=\lbrace u+te:\, u\in {\mathcal C}_-,\, t\geq 0,\, \|u+te\|\leq r_- \rbrace,
\]
\[
H=\lbrace u+te:\, u\in {\mathcal C}_-,\, t\geq 0,\, \|u+te\|= r_- \rbrace,
\]
and assume that $r_->r_+$. Then the following facts hold:
\begin{itemize}
\item[(c)] $(Q,D_-\cup H)$ links $S_+$ cohomologically in dimension $m+1$ over $\mathbb{K}$;

\item[(d)]  $D_-\cup H$ links $(D_+,S_+)$ cohomologically in dimension $m$ over $\mathbb{K}$;
\end{itemize}
\end{theorem}

In order to prove our existence result, we shall use assertion (c) in Section \ref{seclinking} and assertion (a) in Section \ref{secsaddle}, that correspond to the classical linking and saddle geometry, respectively.

We will also take advantage of the following result
\begin{cor}[\cite{dela}, Corollary 2.9]\label{corgeom}
Let $X$ be a real normed space and let ${\mathcal C}_-,{\mathcal C}_+$ be two symmetric cones in $X$ such that ${\mathcal C}_+$ is closed in $X$, 
${\mathcal C}_-\cap {\mathcal C}_+ = \lbrace 0 \rbrace$ and such that 
\[
i({\mathcal C}_-\setminus \lbrace0\rbrace)= i(X\setminus {\mathcal C}_+)<\infty.
\]
Then the assertion (a)-(d) of Theorem $\ref{geom}$ hold for $m=i({\mathcal C}_-\setminus \lbrace0\rbrace)$ and $\mathbb{K}=\mathbb{Z}_2$.
\end{cor}

Going back to definitions \eqref{c-} and \eqref{c+}, we have the following result, which is the transcription in our setting of \cite[Theorem 3.2]{dela}, and whose proof follows that one step-by-step.
\begin{theorem}\label{index}
Let $m\geq 1$ be such that $\lambda_m < \lambda_{m+1}$, then we have
\[
i(C_m^- \setminus \lbrace 0 \rbrace) =i(X\setminus C_m^+)=m
\]
\end{theorem}

Finally, in order to use Theorem \ref{critp}, the crucial tool is
\begin{prop}[\cite{dela}, Proposition 2.4]\label{collegamento}
If $(D, S)$ links $(A, B)$ cohomologically (in some dimension), then $(D, S)$ links $(A, B)$.
\end{prop}

\section{Linking-like problems}\label{seclinking}
Now, let us go back to problem \eqref{plink}, that is
\[
\begin{cases}
(-\Delta)^s_p u =\lambda |u|^{p-2}u + g(x,u) & $ in $ \Omega,
\\
\mathscr{N}_{s,p}u=0  & $ in $ \R^N \setminus \overline{\Omega}.
\end{cases}
\]
We recall that $p\in(1,\infty)$, $\Omega$ is a bounded domain with Lipschitz boundary, $\lambda \geq 0$ and $g:\Omega\times \R \to \R$ is a Carath\'eodory function, that is the map $x\mapsto g(x,t)$ is measurable for every $t\in \R$ and the map $t\mapsto g(x,t)$ is continuous for a.e. $x\in \Omega$. 

Of course, we shall assume growth conditions on $g$ which will ensure that any critical point of the $C^1$ functional $I:X\to \R$ defined as
\begin{equation}\label{I}
I(u)=\frac{1}{2p}\int\int_\Q\frac{|u(x)-u(y)|^p}{|x-y|^{N+ps}}\, dxdy-\frac{\lambda}{p}\int_\Omega |u|^p\,dx 
-\int_\Omega G(x,u)\,dx
\end{equation}
is a weak solution of \eqref{plink}.

\begin{rem}
Notice that, quite strangely, the coefficient $\frac{1}{2}$ appears in front of the expected $\frac{1}{p}[u]^p$. This is related to symmetry properties of the double integral in the definition of $I$, and it justifies the fact that $u$ solves \eqref{plink} if and only if $I'(u)=0$, see \cite{BMPS,mupli}.
\end{rem}

We first we give the following result, which will be useful in any case and which makes precise the statement in \cite{mupli} related to the $(S)$ property.
\begin{prop}\label{s+}
Set $A(u)=[u]^p$.
Then the functional $A':X\to X'$  
satisfies the $(S)_+$ property, that is for every sequence $(u_n)_n$ such that $u_n \rightharpoonup u$ in $X$  as $n\to \infty$ and
\begin{equation}\label{lsup}
\limsup_{n\to \infty}\langle A'(u_n),u_n-u\rangle_{X',X} \leq 0,
\end{equation}
then $u_n \to u$ in $X$ as $n\to \infty$.
\end{prop}
\begin{proof}
Assume that $u_n \rightharpoonup u$ in $X$ and $\limsup \langle A'(u_n),u_n-u\rangle_{X',X} \leq 0$.
First of all, $A$ is convex, of class $C^1$ and weakly lower semicontinuous in $X$, so that
$A(u)\leq \liminf A(u_n)$.

Moreover, the linear functional $\langle A'(u),\cdot \rangle_{X',X}$ is in $X'$. So, since $u_n \rightharpoonup u$ in $X$,
\begin{equation}\label{A'to0}
\langle A'(u),u_n-u \rangle_{X',X} \to 0
\end{equation}
as $n\to \infty$. By the convexity of $A$, we get that $A'$ is a monotone operator, so that
\[
\langle A'(u_n)-A'(u),u_n-u\rangle_{X',X}\geq0.
\]
By \eqref{lsup}  we get
\[
0\leq \limsup_{n\to \infty} \langle A'(u_n)-A'(u),u_n-u\rangle_{X',X}\leq 0,
\]
and so 
\begin{equation}\label{Alim0}
\lim_{n\to \infty} \langle A'(u_n)-A'(u),u_n-u\rangle_{X',X}=0.
\end{equation}
Hence,  \eqref{A'to0} and \eqref{Alim0} imply that
\begin{equation}\label{Alim02}
\lim_{n\to \infty} \langle A'(u_n),u_n-u\rangle_{X',X}=0.
\end{equation}
Again by the convexity of $A$  we have that
\[
A(u)\geq \langle A'(u_n),u-u_n\rangle_{X',X}\geq A(u_n).
\]
By \eqref{Alim02},
$A(u)\geq \limsup A(u_n)$, and so 
\[
A(u)=\lim_{n\to \infty}A(u_n).
\]
By the compact embedding of $X$ into $L^p(\Omega)$ we also have $u_n \to u$ in $L^p(\Omega)$.
In the end, $\|u_n\|\to \|u\|$.
Hence, by the uniform convexity of $X$ (recall that $1<p<\infty$) , we obtain that $u_n$ converges strongly to $u$ in $X$  as $n\to \infty$.
\end{proof}

\subsection{With the Ambrosetti-Rabinowitz condition}
This case is the easy one, which we present just to show the extension of the approach in \cite{dela} to the nonlocal case.

Here we will further assume the following hypotheses on $g$:
\begin{itemize}
\item[($g_1$)] there exist constants $a_1,a_2>0$ and $q>p$ such that for every $t\in \R$ and for a.e. $x\in \Omega$
\[
|g(x,t)|\leq a_1+a_2 |t|^{q-1},
\]
where $q<\frac{pN}{N-ps}$ if $N>ps$;

\item[($g_2$)] $g(x,t)=o(|t|^{p-1})$ as $t \to 0$ uniformly a.e. in $\Omega$;

\item[($g_3$)] denoting $G(x,t)=\int_0^t g(x,\tau)\, d\tau $, there exist $\mu>p$ and $R\geq 0$ such that for every $t$
with $|t|>R$ and for a.e. $x\in \Omega$
\[
0<\mu G(x,t) \leq g(x,t)t,
\]
and there exist $\tilde \mu>p$, $a_3>0$ and  $a_4\in L^1(\Omega)$ such that for every $t\in \R$ and a.e. $x\in \Omega$,
\begin{equation}\label{add}
G(x,t)\geq a_3|t|^{\tilde \mu} -a_4(x);
\end{equation}

\item[($g_4$)] if $R>0$, then $G(x,t)\geq 0$ for every $t\in \R$ and a.e. $x\in \Omega$.
\end{itemize}

\begin{rem}
Condition \eqref{add} was introduced in \cite{addendum} to complete the Ambrosetti-Rabinowitz condition in presence of a Carath\'eodory functions.
\end{rem}

Our first existence result is
\begin{theorem}\label{teolink}
If hypotheses $(g_1)-(g_4)$ hold, then problem \eqref{plink} admits a nontrivial weak solution.
\end{theorem}

In order to prove Theorem \ref{teolink} it will be enough to apply Theorem \ref{critp}  to the functional $I$ defined in \eqref{I} under the validity of the Palais-Smale condition (of course, if the Cerami condition holds, the Palais-Smale condition holds, as well); hence, we will apply Theorem \ref{critp} in the version of \cite[Theorem 2.2]{dela}, where the Palais-Smale condition is assumed.

Thus, now we prove that $I$ satisfies the Palais-smale condition at any level $c\in\R$ - $(PS)_c$ for short -, that is
\begin{center}
for every sequence $(u_n)_n$ in $X$ such that $I(u_n)\to c$ and $I'(u_n)\to 0$ in $X'$,
there exists a strongly converging subsequence of $(u_n)_n$.
\end{center}

\begin{prop}\label{pslink}
Under the assumptions of Theorem $\ref{teolink}$, $I$ satisfies $(PS)_c$ for every $c\in \R$.
\end{prop}

\begin{proof}
Let $(u_n)_n$ in $X$ be such that $I(u_n)\to c$ and $I'(u_n)\to 0$ and fix $k\in (p,\mu)$.  We re-write the functional in the following way:
\[
\begin{aligned}
I(u)&=\frac{1}{2p}\int\int_\Q\frac{|u(x)-u(y)|^p}{|x-y|^{N+ps}}\, dxdy+\frac{1}{2p}\int_\Omega |u|^pdx\\
&-\left(\frac{\lambda}{p}+\frac{1}{2p}\right)\int_\Omega |u|^p\,dx 
-\int_\Omega G(x,u)\,dx\\
&=\frac{1}{2p}\|u\|^p-\left(\frac{\lambda}{p}+\frac{1}{2p}\right)\int_\Omega |u|^p\,dx 
-\int_\Omega G(x,u)\,dx.
\end{aligned}
\]

We observe that 
\begin{equation}\label{psbound}
k I(u_n)- \langle I'(u_n),u_n \rangle \leq M  + N \|u_n\|
\end{equation}
for some $M,N>0$ and all $n\in \N$. On the other hand, by $(g_3)$ and $(g_1)$ we have
\begin{align*}
k&I(u_n)-\langle I'(u_n),u_n \rangle \\
&=\left(\frac{k}{2p}-\frac{1}{2}\right)\|u_n\|^p-\left(\frac{k}{p}-1\right)\left(\lambda+\frac{1}{2}\right)\int_\Omega |u_n|^p\,dx \\
&+\int_\Omega \big(g(x,u_n)u_n-kG(x,u_n)\big)\,dx \\
&\geq \left(\frac{k}{2p}-\frac{1}{2}\right)\|u_n\|^p-\left(\frac{k}{p}-1\right)\left(\lambda+\frac{1}{2}\right)\int_\Omega |u_n|^p\,dx \\
&+(\mu -k)\int_\Omega G(x,u_n)\,dx -C_R
\end{align*}
for some constant $C_R\geq 0$. By \eqref{add}, we get
\begin{align*}
k&I(u_n)-\langle I'(u_n),u_n \rangle \\
&\geq \left(\frac{k}{2p}-\frac{1}{2}\right)\|u_n\|^p-\left(\frac{k}{p}-1\right)\left(\lambda+\frac{1}{2}\right)\int_\Omega |u_n|^p\,dx \\
&+(\mu -k)a_3\int_\Omega |u_n|^{\tilde\mu} \,dx -C
\end{align*}
for some constant $C\geq 0$. By the H\"older and the Young inequalities, we get that for any $\varepsilon>0$ we have that for every $u\in X$
\[
\|u\|_p^p \leq \varepsilon \|u\|_{\tilde\mu}^{\tilde\mu} + C_\varepsilon.
\]
Thus, we obtain
\begin{align*}
k&I(u_n)-\langle I'(u_n),u_n \rangle \\
&\geq \left(\frac{k}{2p}-\frac{1}{2}\right)\|u_n\|^p
+\left[(\mu -k)a_3- \varepsilon \left( \frac{k}{p}-1\right)\left(\lambda+\frac{1}{2}\right)\right] \int_\Omega |u_n|^{\tilde \mu}\,dx -\tilde{C_\varepsilon}
\end{align*}
for some $\tilde{C_\varepsilon}>0$.
Taking $\varepsilon$ small enough, we get
\[
kI(u_n)-\langle I'(u_n),u_n \rangle \geq \left(\frac{k}{2p}-\frac{1}{2}\right)\|u_n\|^p -\tilde{C_\varepsilon}.
\]
This together with \eqref{psbound} implies that $(u_n)_n$ is bounded in $X$. Up to a subsequence, we can assume that
$u_n\rightharpoonup u$ in $X$ and $u_n \to u$ in $L^p(\Omega)$ as $n\to \infty$. By assumption, we have
\[
\langle I'(u_n),u_n-u \rangle \to 0.
\]
On the other hand
\begin{align*}
\langle& A'(u_n),u_n-u \rangle \\
&=\langle I'(u_n),u_n-u \rangle
+\lambda\int_\Omega |u_n|^{p-2}u_n(u_n-u)\,dx+\int_\Omega g(x,u_n)(u_n-u)\,dx.
\end{align*}
Since $u_n \to u$ in $L^p(\Omega)$, from ($g_1$) we obtain that
\[
\int_\Omega |u_n|^{p-2}u_n(u_n-u)\,dx \to 0 
\]
and
\[
\int_\Omega g(x,u_n)(u_n-u)\,dx \to 0;
\]
so $\langle A'(u_n),u_n-u \rangle_{X',X} \to 0$ as $n\to \infty$.
By Proposition \ref{s+} we get that  $u_n \to u$ in $X$, as desired.
\end{proof}

Now we are ready to prove Theorem \ref{teolink}.

\begin{proof}
Let $(\lambda_m)_m$ be the sequence of eigenvalues defined in \eqref{lm}.
Since this sequence is divergent, there exists $m\geq 1$ such that $\lambda_m \leq 2\lambda +1< \lambda_{m+1}$. 
Defining $C_m^-$ and $C_m^+$ as in \eqref{c-} and \eqref{c+}, we have that $C_m^-$,$C_m^+$ are two symmetric closed cones in $X$
with $C_m^-\cap C_m^+ =\lbrace0\rbrace$. We recall that by Theorem \ref{index} we have
\[
i(C_m^-\setminus \lbrace0\rbrace)= i(X\setminus C_m^+)=m.
\]

Now, by $(g_1)$ and $(g_2)$ it is standard to see that for any $\ve>0$ there exists $C_\ve>0$ such that
\[
|G(x,t)|\leq \frac{\ve}{2p} |t|^p+C_\ve|t|^q
\]
for a.e. $x\in \Omega$ and all $t\in \R$. As a consequence, taking $u\in C_m^+$, by the inequality in \eqref{c+} and the Sobolev inequality, we have that
\[
\begin{aligned}
I(u)&\geq\frac{1}{2p}\|u\|^p -\frac{2\lambda+1}{2p}\int_\Omega |u|^pdx-\frac{\ve}{2p} \int_\Omega |u|^pdx-C_\ve \int_\Omega |u|^qdx\\
&\geq\frac{1}{2p}\|u\|^p -\frac{1}{2p\lambda_{m+1}}\left(2\lambda+1+\ve\right)[u]^p-C_\ve \int_\Omega |u|^qdx\\
& \geq\frac{1}{2p}\left(1-\frac{2\lambda+1+\ve}{\lambda_{m+1}}\right)\|u\|^p -C\|u\|^q
\end{aligned}
\]
for some $C>0$.

Hence, choosing $\ve$ small enough, there exists $r_+>0$ and $\alpha>0$ such that, if $\|u\|=r_+$, then $I(u)\geq \alpha$. 

On the other hand, taking $u\in C_m^-$, $e\in X\setminus C_m^-$ and $t>0$, by \eqref{add} we get that
\begin{align*}
I(u+te)&\leq \frac{2^{p-2}}{p}\left( \int\int_\Q\frac{|u(x)-u(y)|^p}{|x-y|^{N+ps}}\, dxdy 
+t^p\int\int_\Q\frac{|e(x)-e(y)|^p}{|x-y|^{N+ps}}\, dxdy \right)\\
&-\frac{\lambda}{p} \int_\Omega |u+te|^p\,dx -a_3t^{\tilde \mu}\int_\Omega \left|\frac{u}{t}+e\right|^{\tilde\mu} \,dx+\|a_4\|_1 \to -\infty
\end{align*} 
as $t \to +\infty$.
In conclusion, there exists $r_->r_+$ such that $I(v)\leq 0$ when $v\in C_m^-+(\R^+e)$ and $\|v\|\geq r_-$.

Defining $D_-,S_+,Q$ and $H$ as in Theorem \ref{geom}, by Corollary \ref{corgeom} we have that $(Q,D_-\cup H)$ links $S_+$
cohomologically in dimension $m+1$ over $\mathbb{Z}_2$. In particular, $(Q,D_-\cup H)$ links $S_+$ by Proposition \ref{collegamento}.
In addition, $I$ is bounded on $Q$, $I(u)\leq 0$ for every $u\in D_-\cup H$ and $I(u)\geq \alpha>0$ for every $u\in S_+$.
By Proposition \ref{pslink} $(PS)_c$ holds.
Finally, by applying Theorem \ref{critp} with $S=D_-\cup H$, $D=Q$, $A=S_+$ and $B=\emptyset$, $I$ admits a critical value $c\geq \alpha$, hence there exists a critical point $u$ with $I(u)=c>0$.
It follows that $u$ is a nontrivial weak solution of \eqref{plink}.
\end{proof}

\subsection{Without the Ambrosetti-Rabinowitz condition}

In this section we consider the problem
 
\begin{equation}\label{plinkc}
\begin{cases}
(-\Delta)^s_p u =\lambda |u|^{p-2}u + f(x,u) \quad $ in $ \Omega,
\\
\mathscr{N}_{s,p}u=0  \quad \quad $ in $ \R^N \setminus \overline{\Omega},
\end{cases},
\end{equation}
where $\lambda \geq 0$ and $f:\Omega\times \R \to \R$ is a Carath\'eodory function such that $f(x,0)=0$ for almost every
$x\in \Omega$. This time, we assume the following hypotheses on $f$, first introduced in \cite{MP}:
\begin{itemize}
\item[$(f_1)$] there exists $a\in L^q(\Omega)$, $a\geq 0$, with $q\in ((p^*_s)',p)$, $c>0$ and $r\in (p,p^*_s)$ such that
$$|f(x,t)|\leq a(x)+c|t|^{r-1} $$
for a.e. $x \in \Omega$ and for all $t \in \R$;
\item[$(f_2)$] denoting $F(x,t)=\int_0^t f(x,\tau)d\tau $, we have 
$$\lim_{t\to \pm \infty}\frac{F(x,t)}{|t|^p}=+\infty $$
uniformly for a.e. $ x\in \Omega$;
\item[$(f_3)$] if $\sigma(x,t):=f(x,t)t-pF(x,t)$, then there exist $\vartheta\geq1$ and $\beta^* \in L^1(\Omega)$, $\beta^*\geq0$, such that
\[
\sigma(x,t_1)\leq \vartheta\sigma(x,t_2)+\beta^*(x)
\]
for a.e. $x\in \Omega$ and all $0\leq t_1 \leq t_2$ or $t_2\leq t_1 \leq 0$;
\item[$(f_4)$] 
$$\lim_{t\to 0} \frac{f(x,t)}{|t|^{p-2}t}=0 $$
uniformly for a.e. $x\in \Omega$.
\end{itemize}
In $(f_1)$ we have denoted by $p^*_s$ the fractional Sobolev exponent of order $s$, that is
\[
p^*_s=\begin{cases}
\dfrac{pN}{N-ps}& \mbox{ if }ps<N,\\
\infty &\mbox{ if }ps\geq N.
\end{cases}
\]
In this way, the embedding in $L^q(\Omega)$ of $W^{s,p}(\Omega)$ (and thus of $X$) is compact for every $q<p^*_s$.

As before, we give the definition of a weak solution.
\begin{dfn}
Let $u\in X$. We say that $u$ is a weak solution of problem \eqref{plink} if
\[
\frac{1}{2}\int\int_\Q \frac{J_p(u(x)-u(y))(v(x)-v(y))}{|x-y|^{N+ps}}\, dxdy
= \lambda\int_\Omega |u|^{p-2}uv\,dx+\int_\Omega f(x,u)v\,dx
\]
for every $v\in X$.
\end{dfn}
Again, any critical point of the $C^1$ functional $\E:X\to \R$ defined as
\[
\E(u)=\frac{1}{2p}\int\int_\Q\frac{|u(x)-u(y)|^p}{|x-y|^{N+ps}}\, dxdy-\frac{\lambda}{p}\int_\Omega |u|^p\,dx 
-\int_\Omega F(x,u)\,dx
\]
is a weak solution of \eqref{plink}.

The main result of this section is the following.

\begin{theorem}\label{thlc}
If hypotheses $(f_1)$-$(f_4)$ hold, then problem \eqref{plinkc} admits two nontrivial
constant sign solutions. More precisely, one solution is strictly positive and the other one is strictly negative in $\R^N$.
\end{theorem}

First of all, we introduce the functionals
\begin{align*}
\E_\pm(u)&=\frac{1}{2p}\int\int_\Q\frac{|u(x)-u(y)|^p}{|x-y|^{N+ps}}\, dxdy
+\frac{1}{p}\int_\Omega |u|^pdx  \\
&-\frac{\lambda+1}{p}\int_\Omega |u^\pm|^pdx -\int_\Omega F(x,u^\pm)\,dx ,
\end{align*}
where $u^+:=\max\{u,0\}$ and $u^-:=\max\{-u,0\}$ are the classical positive part and negative part of $u$, respectively. Notice that $\E_+(u)=\E(u)$ for every $u\geq0$ and $\E_-(u)=\E(u)$ for every $u\leq0$.

The following algebraic inequalities will be very useful in the following:
\begin{equation}\label{disug}
|x^--y^-|^p \leq |x-y|^{p-2}(x-y)(y^--x^-),
\end{equation}
\begin{equation}\label{disug1}
|x^+-y^+|^p \leq |x-y|^{p-2}(x-y)(x^+-y^+),
\end{equation}
\begin{equation}\label{disug2}
|x-y|^p\leq 2^{p-1}(|x^+-y^+|^p+|x^--y^-|^p)
\end{equation}
and
\begin{equation}\label{11}
|x^\pm-y^\pm|\leq |x-y|
\end{equation}
for any $x,y\in \R$. The proofs are obvious.

\begin{prop}\label{C}
Under the assumptions of Theorem $\ref{thlc}$, $\E_\pm$ satisfies $(C)_c$ for every $c\in \R$.
\end{prop}
\begin{proof}
We do the proof for ${\mathscr E}_+$, the proof for ${\mathscr E}_-$ being analogous.

Let $(u_n)_n$ in $X$ be such that 
\begin{equation}\label{cer1}
|{\mathscr E}_+(u_n)|\leq M_1
\end{equation}
for some $M_1>0$ and all $n\geq 1$, and 
\begin{equation}\label{cer2}
(1+\|u_n\|){\mathscr E}_+'(u_n)\to 0
\end{equation}
in $X'$ as $n\to \infty$.
From \eqref{cer2} we have
$$|{\mathscr E}_+'(u_n)(h)|\leq \frac{\varepsilon_n\|h\|}{1+\|u_n\|} $$
for every $h\in X$ and with $\varepsilon_n \to 0$ as $n \to \infty$, that is 
\begin{equation}\label{vs}
\begin{aligned}
\left|\frac{1}{2}\int \int_\Q\right.&\frac{J_p(u_n(x)-u_n(y))(h(x)-h(y))}{|x-y|^{N+ps}}\,dxdy
+\int_\Omega |u_n|^{p-2}u_nh\,dx \\
&- (\lambda +1) \int_\Omega |u_n^+|^{p-2}u_n^+ h\,dx \left. -\int_\Omega f(x,u_n^+)h\,dx\right|
\leq \frac{\varepsilon_n\|h\|}{1+\|u_n\|}. 
\end{aligned}
\end{equation}
Taking $h=-u_n^-$ in \eqref{vs},  we obtain
\begin{equation}\label{a0}
\frac{1}{2}\int \int_\Q\frac{J_p(u_n(x)-u_n(y))(u_n^-(y)-u_n^-(x))}{|x-y|^{N+ps}}\,dxdy +\lambda\int_\Omega |u_n^-|^pdx\leq \varepsilon_n,
\end{equation}
and by \eqref{disug} we get
\[
\int \int_\Q\frac{|u_n^-(x)-u_n^-(y)|^p}{|x-y|^{N+ps}}\,dxdy+2\lambda\int_\Omega |u_n^-|^pdx \leq 2\varepsilon_n.
\]
As a consequence, we get that
\begin{equation}\label{vs-}
u_n^-\to 0 \mbox{ in $X$ as $n\to \infty$}.
\end{equation}
In particular, $(u_n^-)_n$ is bounded in $X$.

On the other hand, taking $h=-u_n^+$ in \eqref{vs}, we get
\begin{equation}\label{vs+}
\begin{aligned}
-\frac{1}{2}\int \int_\Q &\frac{J_p(u_n(x)-u_n(y))(u_n^+(x)-u_n^+(y))}{|x-y|^{N+ps}}\,dxdy \\
&+ \lambda \int_\Omega |u_n^+|^p\,dx
 +\int_\Omega f(x,u_n^+)u_n^+\,dx\leq \varepsilon_n.
\end{aligned}
\end{equation}

From \eqref{cer1} we know that
\begin{equation}\label{maggio}
\frac{1}{2}[u_n]^p + \int_\Omega |u_n|^p\,dx
- (\lambda+1) \int_\Omega |u_n^+|^p\,dx
 -p\int_\Omega F(x,u_n^+)\,dx\leq pM_1
\end{equation}
for all $n\geq 1$. Now, by \eqref{a0} and \eqref{vs-}, we have that
\[
\int \int_\Q\frac{J_p(u_n(x)-u_n(y))(u_n^-(x)-u_n^-(y))}{|x-y|^{N+ps}}\,dxdy \to 0,
\]
and so from \eqref{maggio} we get
\begin{equation}\label{m2}
\begin{aligned}
\frac{1}{2}\int \int_\Q &\frac{J_p(u_n(x)-u_n(y))(u_n^+(x)-u_n^+(y))}{|x-y|^{N+ps}}\,dxdy \\
&+ \int_\Omega |u_n|^p\,dx - (\lambda+1) \int_\Omega |u_n^+|^p\,dx
 -p\int_\Omega F(x,u_n^+)\,dx\leq M_2
 \end{aligned}
\end{equation}
for some $M_2>0$ and all $n\geq 1$.
Adding \eqref{m2} to \eqref{vs+} we obtain
\[
 \int_\Omega |u_n|^p\,dx
- \int_\Omega |u_n^+|^p\,dx +\int_\Omega f(x,u_n^+)u_n^+ \,dx -p \int_\Omega F(x,u_n^+)\,dx \leq M_3
\] 
for some $M_3>0$ and all $n\geq 1$, which clearly implies
\begin{equation}\label{m3}
\int_\Omega \sigma(x,u_n^+)\,dx \leq M_3.
\end{equation}

Now we claim that $(u_n^+)_n$ is bounded in $X$, as well. We argue by contradiction. Up to a subsequence,
we assume that $\|u_n^+\|\to \infty$ as $n\to \infty$. Defining $y_n=u_n^+/\|u_n^+\|$, we can assume that
\begin{equation}\label{cdeb}
y_n \rightharpoonup y \quad \text{in } X \quad \text{and } y_n\to y \quad \text{in } L^q(\Omega)
\end{equation}
for every $q\in (p,p^*_s)$ with $y\geq 0$ in $\Omega$.

First we deal with the case $y\not \equiv 0$.
We define $Z(y)=\lbrace x\in\Omega : y(x)=0\rbrace$, and so we have 
$\left|\Omega \setminus Z(y)\right|>0$ and $u_n^+\to \infty$ for almost every $x\in \Omega \setminus Z(y)$ as $n\to \infty$.
By ($f_2$), we have
$$\frac{F(x,u_n^+(x))}{\|u_n^+\|^p}=\frac{F(x,u_n^+(x))}{u_n^+(x)^p}y_n(x)^p \to \infty $$
for almost every $x\in \Omega \setminus Z(y)$. From Fatou's Lemma we get that
$$\int_\Omega \liminf_{n\to \infty} \frac{F(x,u_n^+(x))}{\|u_n^+\|^p}\,dx \leq 
\liminf_{n\to \infty}\int_\Omega \frac{F(x,u_n^+(x))}{\|u_n^+\|^p}\, dx,  $$
and so 
\begin{equation}\label{fatu}
\int_\Omega \frac{F(x,u_n^+(x))}{\|u_n^+\|^p}\, dx \to \infty
\end{equation}
as $n\to \infty$.

Again from \eqref{cer1}  we have
$$-\frac{1}{2p}[u_n]^p -\frac{1}{p} \int_\Omega |u_n|^p\,dx
+ \frac{\lambda +1}{p} \int_\Omega |u_n^+|^p\,dx+\int_\Omega F(x,u_n^+)\,dx \leq M_4$$
for some $M_4>0$ and $n\geq 1$. From \eqref{disug2} we get
\[
-\frac{2^{p-2}}{p}([u_n^+]^p+[u_n^-]^p)-\frac{1}{p} \int_\Omega |u_n|^p\,dx 
+ \frac{\lambda +1}{p} \int_\Omega |u_n^+|^p\,dx +\int_\Omega F(x,u_n^+)\,dx \leq M_4,
\]
and from \eqref{vs-}
\[
-\frac{2^{p-2}}{p}[u_n^+]^p
+ \frac{\lambda }{p} \int_\Omega |u_n^+|^p\,dx +\int_\Omega F(x,u_n^+)\,dx \leq M_5,
\]
for some $M_5>0$ and all $n\geq1$, so that
\begin{align*}
\int_\Omega F(x,u_n^+)\,dx \leq M_5+c\|u_n^+\|^p
\end{align*}
for some $c>0$ and all $n\geq1$. Dividing by $\|u_n^+\|^p$ and passing to the limit we obtain
\[
\limsup_{n\to \infty}\int_\Omega \frac{F(x,u_n^+(x))}{\|u_n^+\|^p}\, dx \leq M_6  
\]
for some $M_6$, which is in contradiction with \eqref{fatu}, and this concludes the case $y\neq 0$.

Now, we deal with the case $y\equiv 0$. We consider the continuous functions $\gamma_n:[0,1]\to\R$, defined as
\[
\gamma_n(t):={\mathscr E}_+(tu_n^+)
\]
for any $n\geq 1$. So, there exists $t_n\in[0,1]$ such that
\begin{equation}\label{max}
\gamma_n(t_n)=\max_{t\in[0,1]}\gamma_n(t).
\end{equation} 
Now, fixed $\mu>0$,  we define $v_n:=(p\mu)^\frac{1}{p}y_n\in X$. From \eqref{cdeb} we get that $v_n \to 0$ 
in $L^q(\Omega)$ for all $q\in (p,p^*_s)$. From ($f_1$) we know that
$$\int_\Omega F(x,v_n(x))\,dx \leq \int_\Omega a(x)|v_n(x)|\,dx + C\int_\Omega |v_n(x)|^r\,dx, $$
and so 
\begin{equation}\label{to0}
\int_\Omega F(x,v_n(x))\,dx \to 0
\end{equation}
as $n\to \infty$.
Since $\|u_n^+\|\to \infty$, there exists $n_0\geq 1$ such that 
$(p\mu)^\frac{1}{p} /\|u_n^+\| \in(0,1)$ for all $n\geq n_0$. Then, from \eqref{max}, we have
$$\gamma_n(t_n)\geq \gamma_n\left(\frac{(p\mu)^\frac{1}{p}}{\|u_n^+\|} \right) $$
for all $n\geq n_0$.
Thus, we get
\begin{align*}
&{\mathscr E}_+(t_nu_n^+) \geq {\mathscr E}_+((p\mu)^\frac{1}{p}y_n)={\mathscr E}_+(v_n) \\
& =\frac{1}{2}\mu  \int \int_\Q \frac{|y_n(x)-y_n(y)|^p}{|x-y|^{N+ps}}\,dxdy 
- \frac{\lambda}{p}  \int_\Omega v_n^p\,dx - \int_\Omega F(x,v_n(x))\,dx  \\
&=\frac{\mu}{2} \|y_n\|^p -\frac{\mu}{2}\int_\Omega y_n^pdx- \frac{2\lambda +1}{2p}  \int_\Omega v_n^p\,dx - \int_\Omega F(x,v_n(x))\,dx\\
&=\frac{\mu}{2} - \frac{2\lambda +1}{2p}  \int_\Omega v_n^p\,dx - \int_\Omega F(x,v_n(x))\,dx
\end{align*} 
From \eqref{to0} and the fact that $v_n\to 0$ in $L^p(\Omega)$, we get that
$$
{\mathscr E}_+(t_nu_n^+)\geq \frac{\mu}{2} + o(1),
$$ 
where $o(1)\to 0$ as $n\to \infty$. Since $\mu$ is arbitrary, we have
\begin{equation}\label{toinf}
\lim_{n\to \infty}{\mathscr E}_+(t_nu_n^+)=+ \infty.
\end{equation} 
 
On the other hand, since $0\leq t_nu_n^+\leq u_n^+$ for all $n\leq 1$, from ($f_3$) we get
\begin{equation}\label{sigma}
\int_\Omega \sigma(x,t_nu_n^+)\,dx \leq \vartheta \int_\Omega \sigma(x,u_n^+)\,dx + \|\beta^*\|_1
\end{equation}
for all $n\geq 1$.

In addition, we have that ${\mathscr E}_+(0)=0$; moreover, from \eqref{11}  we get that
\[
\E_+(u_n^+)\leq \E_+(u_n)\leq M_1
\]
for all $n\ge1$ by \eqref{cer1}. Together with \eqref{toinf}, these two facts imply the existence of $ n_1\geq n_0$ such that  $t_n\in (0,1)$ for all
$n\geq n_1$, namely $t_n\neq0$ and $t_n\neq1$.
Since $t_n$ is a maximum point for $\gamma_n$, we have 
\begin{equation}\label{zero}
\begin{aligned}
0&= t_n\gamma_n'(t_n) \\
&=\frac{1}{2}\int \int_\Q \frac{|t_nu_n^+(x)-t_nu_n^+(y)|^p}{|x-y|^{N+ps}}\,dxdy \\
&- \lambda \int_\Omega |t_nu_n^+|^p\,dx- \int_{\Omega} f(x,t_nu_n^+(x))t_nu_n^+(x)\,dx.
\end{aligned}
\end{equation}
Adding \eqref{zero} to \eqref{sigma}, we get
\begin{align*}
\frac{1}{2}\int \int_\Q& \frac{|t_nu_n^+(x)-t_nu_n^+(y)|^p}{|x-y|^{N+ps}}\,dxdy \\
&- \lambda \int_\Omega |t_nu_n^+|^p\,dx - p \int_\Omega F(x,t_nu_n^+(x))\,dx \\ 
&\leq \vartheta \int_\Omega \sigma(x,u_n^+)\,dx + \|\beta^*\|_1,
\end{align*}
which is 
$$p{\mathscr E}_+(t_nu_n^+)\leq  \vartheta \int_\Omega \sigma(x,u_n^+)\,dx + \|\beta^*\|_1.$$
So, from \eqref{toinf}, we get
\begin{equation}\label{toinf2}
\lim_{n\to\infty}\int_\Omega \sigma(x,u_n^+)\,dx = \infty.
\end{equation}
Comparing \eqref{m3} and \eqref{toinf2} we obtain a contradiction, and so the claim follows.

In conclusion, we have proved that $(u_n^+)_n$ is bounded in $X$, so from \eqref{disug2} and \eqref{vs-} we have that $(u_n)_n$ is bounded in $X$.
Hence, we can assume that
\begin{equation}\label{cdeb2}
u_n \rightharpoonup u \quad \text{in } X \quad \text{and } u_n\to u \quad \text{in } L^q(\Omega)
\end{equation}
for every $q\in (p,p^*_s)$ as $n\to \infty$. Taking $h=u_n-u$ in \eqref{vs}, we have
\begin{equation}\label{vss}
\begin{aligned}
&\left|\frac{1}{2}\int \int_\Q  \frac{|u_n(x)-u_n(y)|^p}{|x-y|^{N+ps}}\,dxdy \right. \\
&-\frac{1}{2}\int \int_\Q \frac{J_p(u_n(x)-u_n(y))(u(x)-u(y))}{|x-y|^{N+ps}}\,dxdy 
 +\int_\Omega |u_n|^{p-2}u_n(u_n-u)  \\
&\left.- (\lambda+1)\int_\Omega |u_n^+|^{p-2}u_n^+(u_n-u) \,dx 
- \int_{\Omega} f(x,u_n^+)(u_n-u)\,dx\right|\leq\varepsilon_n.
\end{aligned}
\end{equation}
From ($f_1$) and \eqref{cdeb2}, we have
\[
\int_{\Omega} f(x,u_n^+(x))(u_n(x)-u(x))\,dx \to 0, 
\]
\[
\int_\Omega |u_n|^{p-2}u_n(u_n-u)\to 0
\]
and 
\[
\int_\Omega |u_n^+|^{p-2}u_n^+(u_n-u)\to 0
\]
as $n\to \infty$. 
Passing to the limit in \eqref{vss}, we get
\begin{align*}
&\int \int_\Q  \frac{|u_n(x)-u_n(y)|^p}{|x-y|^{N+ps}}\,dxdy  \\
&-\int \int_\Q \frac{J_p(u_n(x)-u_n(y))(u(x)-u(y))}{|x-y|^{N+ps}}\,dxdy \to 0
\end{align*}
as $n\to \infty$. From Proposition \ref{s+} we can conclude that $u_n \to u$ in $X$ and this concludes the proof that ${\mathscr E}_+$ satisfies $(C)_c$ for every $c\in \R$.

Proceeding analogously, we have that $\E_-$ satisfies $(C)_c$ for every $c\in \R$, as well.
\end{proof}

Now we are ready to give the proof Theorem \ref{thlc}.

\begin{proof}[Proof of Theorem $\ref{thlc}$]
First, we want to apply Theorem \ref{critp} to $\E_+$.
So, as before, let $(\lambda_m)_m$ be the sequence of eigenvalues defined in \eqref{lm}.
As in the proof of Theorem \ref{teolink}, there exists $m\geq 1$ such that $\lambda_m \leq 2\lambda+1 < \lambda_{m+1}$, and we use the same two symmetric closed cones $C_m^-$ and $C_m^+$
with $C_m^-\cap C_m^+ =\lbrace0\rbrace$.
By Theorem \ref{index} we also have
\[
i(C_m^-\setminus \lbrace0\rbrace)= i(X\setminus C_m^+)=m.
\]
In a similar way to the proof of Theorem \ref{teolink}, by ($f_1$), ($f_4$) and taking $u\in C_m^+$ we have
\[
\begin{aligned}
\E_+(u)&\geq\frac{1}{2p}\|u\|^p -\frac{2\lambda+1}{2p}\int_\Omega |u^+|^pdx
-\frac{\ve}{2p} \int_\Omega |u^+|^pdx-C_\ve \int_\Omega |u^+|^qdx\\
&\geq\frac{1}{2p}\|u\|^p -\frac{2\lambda+1}{2p}\int_\Omega |u|^pdx
-\frac{\ve}{2p} \int_\Omega |u|^pdx-C_\ve \int_\Omega |u|^qdx\\
&\geq\frac{1}{2p}\|u\|^p -\frac{1}{2p\lambda_{m+1}}\left(2\lambda+1+\ve\right)[u]^p-C_\ve \int_\Omega |u|^qdx\\
& \geq\frac{1}{2p}\left(1-\frac{2\lambda+1+\ve}{\lambda_{m+1}}\right)\|u\|^p -C\|u\|^q
\end{aligned}
\]
for some $C>0$. So there exists $r_+>0$ and $\alpha>0$ such that, if $\|u\|=r_+$ then $\E_+(u)\geq \alpha$.

On the other hand, taking $u\in C_m^-$, $e\in X\setminus C_m^-$ with $e^+\neq 0$ and $t>0$, from ($f_2$) we get
\begin{align*}
&\E_+(u+te)  \leq \frac{1}{2p}\|u+te\|^p -\frac{2\lambda+1}{2p}\int_\Omega |(u+te)^+|^pdx -\int_ \Omega F(x,(u+te)^+)\,dx\\
&\leq \frac{1}{2p}\|u+te\|^p \left( 1- \int_ \Omega \frac{F(x,(u+te)^+)}{((u+te)^+)^p} \frac{((u+te)^+)^p}{\|u+te\|^p}\,dx \right)
 \to -\infty
\end{align*} 
as $t \to +\infty$. So, there exists $r_->r_+$ such that $\E_+(u)\leq 0$ when $u\in C_m^-+\R^+e$ and $\|u\|\geq r_-$.

Again, we define $D_-,S_+,Q$ and $H$ as in Theorem \ref{geom}.
By Corollary \ref{corgeom} we have that $(Q,D_-\cup H)$ links $S_+$
cohomologically in dimension $m+1$ over $\mathbb{Z}_2$. In particular, $(Q,D_-\cup H)$ links $S_+$.
In addition, $\E_+$ is bounded on $Q$, $\E_+(u)\leq 0$ for every $u\in D_-\cup H$ and $\E_+(u)\geq \alpha>0$ 
for every $u\in S_+$. Moreover, by Proposition \ref{C} $(C)_c$ holds as well.

By Theorem \ref{critp}, $\E_+$ admits a critical value $c\geq \alpha$, hence a critical point $u$ with $\E_+(u)>0$.
In particular, we have
\begin{align*}
0&=-\frac{1}{2} \int \int_\Q \frac{J_p(u(x)-u(y))(u^-(x)-u^-(y))}{|x-y|^{N+ps}}\,dxdy - \int_\Omega |u|^{p-2}uu^-\,dx\\
&+(\lambda+1) \int_\Omega |u^+|^{p-2}u^+u^-\,dx +\int_ \Omega f(x,u^+)u^-\,dx \\
&= -\frac{1}{2}\int \int_\Q \frac{J_p(u(x)-u(y))(u^-(x)-u^-(y))}{|x-y|^{N+ps}}\,dxdy + \int_\Omega (u^-)^pdx.
\end{align*}
From \eqref{disug} we get
\[
0\geq \int\int_\Q\frac{|u^-(x)-u^-(y)|^p}{|x-y|^{N+ps}}\, dxdy + \int_\Omega (u^-)^pdx
\]
so that $u^-\equiv 0$ and $u\geq 0$. As a consequence, $\E_+(u)=\E(u)$,
and so $u\geq 0$ is a nontrivial solution of \eqref{plinkc}.

Arguing in the same way for $\E_-$, we can find a nontrivial negative solution $v$ for \eqref{plinkc}.

By the maximum principle (see, for instance, \cite{DPRS} and \cite{MPV} for the Robin problem and also \cite{musina} for some linear cases), we can conclude that $u>0$ and $v<0$ a.e. in $\R^N$.
\end{proof}

\section{A problem with linear growth}\label{secsaddle}
In this section we consider the problem
\begin{equation}\label{psad}
\begin{cases}
(-\Delta)^s_p u = g(x,u) \quad $ in $ \Omega,
\\
\mathscr{N}_{s,p}u=0  \quad \quad $ in $ \R^N \setminus \overline{\Omega},
\end{cases},
\end{equation}
where $\Omega$ is as before and $g:\Omega\times \R\to \R$ is a Carath\'eodory function with $p-$linear growth; namely, there exist $a\in L^{p'}(\Omega)$ and $b\in \R$ such that
\begin{equation}\label{gbound}
|g(x,t)|\leq a(x)+b|t|^{p-1}
\end{equation}
for every $t\in \R$ and for a.e. $x\in \Omega$.

As usual, we define the functional 
\[
I(u):=\frac{1}{2p}\int\int_\Q\frac{|u(x)-u(y)|^p}{|x-y|^{N+ps}}\, dxdy -\int_\Omega G(x,u)\,dx
\]  
so that every critical point of $I$ is a weak solution of \eqref{psad}.

In order to state our result, we need to introduce
\begin{equation}\label{asu}
\overline{\alpha}(x):=\limsup_{|t|\to \infty}\frac{g(x,t)}{|t|^{p-2}t}
\end{equation}
for a.e. $x \in \Omega$.
Then we have:
\begin{theorem}\label{thsad}
Assume \eqref{gbound}. 
If $\overline{\alpha}(x)<\lambda_1=0$,
then problem \eqref{psad} admits a weak solution.
\end{theorem}
\begin{proof}
In this case we shall obtain one solution by applying the Weierstrass Theorem to $I$.

First, we claim that 
\begin{equation}\label{lsup1}
\limsup_{|t|\to \infty}\frac{G(x,t)}{|t|^p}\leq \frac{\overline{\alpha}(x)}{p},
\end{equation}
where $G(x,t):= \int_0^t g(x,\tau)\, d\tau$. By \eqref{asu}, for every $\varepsilon>$ there exists $K>0$ such that
\[
\frac{g(x,t)}{t^{p-1}}<\overline{\alpha}(x)+\varepsilon
\]
for $t\geq K$ and a.e. $x\in \Omega$. Reasoning in a similar way for $t<0$ and integrating gives
\[
G(x,t)\leq \frac{\overline{\alpha}(x)+\varepsilon}{p}(|t|^p-K^p)+\max\left\{G(x,K),G(x,-K)\right\}
\]
for $|t|\geq K$. Hence,
\[
\limsup_{|t|\to \infty}\frac{G(x,t)}{|t|^p}\leq \frac{\overline{\alpha}(x)}{p}
\]
as claimed.

Now we want to prove that \eqref{lsup1} implies that
\begin{equation}\label{toinfs}
\liminf_{\|u\|\to \infty}\frac{I(u)}{\|u\|^p}>0.
\end{equation}
Indeed, take a sequence $(u_n)_n$ in $X$ such that $\|u_n\|\to \infty$. Up to a subsequence, we can assume that 
$v_n:=\frac{u_n}{\|u_n\|}$ converges to some $u$ weakly in $X$ and strongly in $L^p(\Omega)$.
Moreover, $\|u\|\leq 1$, and also
\[
\frac{G(x,u_n)}{\|u_n\|^p}\leq \frac{a(x)|u_n|+b|u_n|^p/p}{\|u_n\|^p}\to \frac{b}{p}|u|^p
\]
in $L^1(\Omega)$ as $n\to \infty$. By the generalized Fatou Lemma we have
\begin{equation}\label{FL}
\limsup_{n\to \infty}\int_\Omega \frac{G(x,u_n)}{\|u_n\|^p}\, dx
\leq \int_\Omega \limsup_{n\to \infty}\frac{G(x,u_n)}{\|u_n\|^p}\, dx.
\end{equation}  
If $(u_n(x))_n$ is bounded,
\[
\frac{G(x,u_n(x))}{\|u_n\|^p}\to 0,
\]
while if $|u_n(x)|\to \infty $,
\[
\limsup_{n\to \infty} \frac{G(x,u_n(x))}{\|u_n\|^p}
=\limsup_{n\to \infty} \frac{G(x,u_n(x))}{|u_n(x)|^p} \frac{|u_n(x)|^p}{\|u_n\|^p}\leq \frac{\overline{\alpha}(x)}{p}|u(x)|^p\leq 0.
\]
In both cases 
\[
\limsup_{n\to \infty}\int_\Omega \frac{G(x,u_n)}{\|u_n\|^p}\,dx\leq 0,
\]
but when $u\ne0$, we have
\begin{equation}\label{aggiunta}
\limsup_{n\to \infty}\int_\Omega \frac{G(x,u_n)}{\|u_n\|^p}\,dx<0.
\end{equation}
Therefore, if $u\neq0$ in $\Omega$,  we have
\[
\liminf_{n\to \infty}\frac{I(u_n)}{\|u_n\|^p}\geq -\int_\Omega  \frac{G(x,u_n)}{\|u_n\|^p}\,dx,
\]
and so by \eqref{aggiunta} we get
\begin{equation}\label{aggiunta2}
\liminf_{n\to \infty}\frac{I(u_n)}{\|u_n\|^p}>0.
\end{equation}
On the other hand, if $u\equiv0$ in $\Omega$,
\[
\frac{I(u_n)}{\|u_n\|^p}=\frac{1}{2p}-\int_\Omega\frac{|u_n|^p}{\|u_n\|^p}dx -\int_\Omega\frac{G(x,u_n)}{\|u_n\|^p}dx,
\]
and so \eqref{aggiunta2} holds also in this case.

Since \eqref{aggiunta2} holds for every diverging sequence, \eqref{toinfs} holds, as well.

In conclusion, it is easy to show that $I$ is lower semicontinuous, while it is coercive from \eqref{toinfs}. So we can
apply the Weierstrass Theorem to find a minimum for $I$, which is a solution of problem \eqref{psad}.

\end{proof}

\section*{Acknowledgments}
The first author is Member of the Gruppo Nazionale per l'Analisi Matematica, la Probabilit\`a  a e le loro Applicazioni (GNAMPA) of the Istituto Nazionale di Alta Matematica (INdAM) ``F. Severi''. He is supported by the MIUR National Research Project {\sl Variational methods, with applications to problems in mathematical physics and geometry} (2015KB9WPT\underline{\ }009) and by the FFABR ``Fondo per il finanziamento delle attivit\`a base di ricerca'' 2017.

The second author is is Member of the Gruppo Nazionale per l'Analisi Matematica, la Probabilit\`a  a e le loro Applicazioni (GNAMPA) of the Istituto Nazionale di Alta Matematica (INdAM) ``F. Severi''.

\end{document}